\documentclass[12pt]{amsart}

\usepackage{mathrsfs,amssymb}

\usepackage{mathabx}

\newtheorem{theorem}{Theorem}[section]
\newtheorem{lemma}[theorem]{Lemma}

\theoremstyle{definition}
\newtheorem{definition}[theorem]{Definition}

\newtheorem{con}[theorem]{Conjecture}

\theoremstyle{remark}
\newtheorem{remark}[theorem]{Remark}

\numberwithin{equation}{section}

\let \la=\lambda
\let \e=\varepsilon

\let \a=\alpha
\let \f=\varphi
\let \b=\beta

\let \O=\Omega
\let \si=\sigma

\let \ga=\gamma
\let \D=\Delta

\let \l=\langle
\let \r=\rangle

\begin{document}
\title[Quantitative weighted estimates]
{Quantitative weighted estimates for the Littlewood-Paley square function and Marcinkiewicz multipliers}

\author{Andrei K. Lerner}
\address{Department of Mathematics,
Bar-Ilan University, 5290002 Ramat Gan, Israel}
\email{lernera@math.biu.ac.il}

\thanks{The author was supported by ISF grant No. 447/16 and ERC Starting Grant No. 713927.}

\begin{abstract}
Quantitative weighted estimates are obtained for the Littlewood-Paley square function $S$ associated with a lacunary decomposition of ${\mathbb R}$ and for the Marcinkiewicz multiplier operator.
In particular, we find the sharp dependence on $[w]_{A_p}$ for the $L^p(w)$ operator norm of $S$ for $1<p\le 2$.
\end{abstract}


\maketitle

\section{Introduction}
Given a weight $w$ (i.e., a non-negative locally integrable function on~${\mathbb R}^n$), we say that $w\in A_p, 1<p<\infty,$ if
$$[w]_{A_p}=\sup_Q\,\l w\r_Q\l w^{1-p'}\r_Q^{p-1}<\infty,$$
where the supremum is taken over all cubes $Q\subset {\mathbb R}^n$ and $\langle \cdot\rangle_Q$ is the integral mean over $Q$.

In the recent decade, it has been of great interest to obtain the $L^p(w)$ operator norm estimates (possibly optimal) in terms of $[w]_{A_p}$
for the different operators in harmonic analysis. In particular, it was established that the $L^p(w)$ operator norms for Calder\'on-Zygmund and a large class of Littlewood-Paley
operators are bounded by a multiple of $[w]_{A_p}^{\max\big(1,\frac{1}{p-1}\big)}$ and $[w]_{A_p}^{\max\big(\frac{1}{2},\frac{1}{p-1}\big)}$, respectively, and these bounds are sharp for all $1<p<\infty$
(see \cite{P,H,CMP,L1}).

On the other hand, there are still a number of operators for which the sharp bounds in terms of $[w]_{A_p}$ are not known yet. For example, for rough homogeneous singular integrals $T_{\O}$ with angular part $\O\in L^{\infty}$
the currently best known result says that $\|T_{\O}\|_{L^2(w)\to L^2(w)}$ is at most a multiple of $[w]_{A_2}^2$, and it is an open question whether this bound is sharp (see
\cite{CCPO,HRT,L2}). Several other examples are the main objects of the present paper.

We consider the classical Littlewood-Paley square function associated with a lacunary decomposition of ${\mathbb R}$ and the Marcinkiewicz multiplier operator.
Recall the definitions of these objects. For $k\in {\mathbb Z}$ set $\D_k=(-2^{k+1},-2^k]\cup [2^k,2^{k+1}).$
The Littlewood-Paley square function we shall deal with is defined by
$$Sf=\left(\sum_{k\in {\mathbb Z}}|S_{\D_k}f|^2\right)^{1/2},$$
where $\widehat{S_{\D_k}f}=\widehat{f}\chi_{\D_k}.$ We say that $T_m$ is the Marcinkiewicz multiplier operator if $\widehat{T_mf}=m\widehat{f}$, where $m\in L^{\infty}$ and
$$\sup_{k\in {\mathbb Z}}\int_{\D_k}|m'(t)|dt<\infty.$$

The fact that $S$ and $T_m$ are bounded on $L^p(w)$ for $w\in A_p$ is well known and due to D. Kurtz \cite{K}. Tracking the dependence on $[w]_{A_p}$ in the known proofs yields, for example, that
the $L^2(w)$ operator norms of $S$ and $T_m$ are bounded by a multiple of $[w]_{A_2}^2$ and $[w]_{A_2}^4$, respectively.

In this paper we give new proofs of the $L^p(w)$ boundedness of $S$ and~$T_m$ providing better quantitative estimates it terms of $[w]_{A_p}$. Our main results are the following.

\begin{theorem}\label{lp}
If $\a_p$ is the best possible exponent in
$$\|S\|_{L^p(w)\to L^p(w)}\le C_p[w]_{A_p}^{\a_p},$$
then
$$\max\left(1,\frac{3}{2}\frac{1}{p-1}\right)\le \a_p\le \frac{1}{2}\frac{1}{p-1}+\max\left(1,\frac{1}{p-1}\right)\quad(1<p<\infty);$$
in particular, $\a_p=\frac{3}{2}\frac{1}{p-1}$ for $1<p\le 2$.
\end{theorem}

\begin{theorem}\label{mm}
If $\b_p$ is the best possible exponent in
$$\|T_m\|_{L^p(w)\to L^p(w)}\le C_{p,m}[w]_{A_p}^{\b_p},$$
then
$$\frac{3}{2}\max\left(1,\frac{1}{p-1}\right)\le \b_p\le \frac{p'}{2}+\max\left(1,\frac{1}{p-1}\right)\quad(1<p<\infty).$$
\end{theorem}

Observe that the lower bounds for $\a_p$ and $\b_p$ are immediate consequences of several known results. By a general extrapolation argument due to T. Luque, C. P\'erez and E. Rela \cite{LPR},
if an operator $T$ is such that its unweighted $L^p$ norms satisfy $\|T\|_{L^p\to L^p}\simeq \frac{1}{(p-1)^{\ga_1}}$ as $p\to 1$ and $\|T\|_{L^p\to L^p}\simeq p^{\ga_2}$ as $p\to \infty$, then the best
possible exponent $\xi_p$ in $\|T\|_{L^p(w)\to L^p(w)}\le C[w]_{A_p}^{\xi_p}$ satisfies $\xi_p\ge\max(\ga_2,\frac{\ga_1}{p-1})$. Therefore, the lower bounds for $\a_p$ and $\b_p$ follow from the sharp
unweighted behavior of the $L^p$ norms of $S$ and $T_m$.

Such a behavior for $S$ was found by J. Bourgain \cite{B}:
\begin{equation}\label{as}
\|S\|_{L^p\to L^p}\simeq \frac{1}{(p-1)^{3/2}}\,\,\text{as}\,\, p\to 1\,\,\text{and}\,\, \|S\|_{L^p\to L^p}\simeq p\,\,\text{as}\,\, p\to \infty,
\end{equation}
which implies the lower bound for $\a_p$. These asymptotic relations were obtained in \cite{B} for the circle version of the Littlewood-Paley square function but the
arguments can be transferred to the real line version in a straightforward way. An alternative proof of the first asymptotic relation in (\ref{as}) has been recently found by O. Bakas \cite{Ba}.

The sharp unweighted $L^p$ norm behavior of $T_m$ is due to T. Tao and J. Wright~\cite{TW}:
$$\|T_m\|_{L^p\to L^p}\simeq \max(p,p')^{3/2}\quad(1<p<\infty),$$
which implies the lower bound for $\b_p$.

Bourgain's proof \cite{B} of the first relation in (\ref{as}) was based on a dual restatement in terms of the vector-valued operator
$\sum_{k\in {\mathbb Z}}S_{\D_k}\psi_k$ with its subsequent handling by means of the Chang-Wilson-Wolff inequality~\cite{CWW}.
Our proof of the upper bound for $\a_p$ follows similar ideas but with some modifications.
As the key tool we use Theorem \ref{discr}, which is a discrete analogue of the sharp weighted continuous square function estimate
proved by M.~Wilson~\cite{W1}. Notice that the latter estimate is also based on the Chang-Wilson-Wolff inequality. We mention  that the
sharp $L^2(w)$ bound in Theorem \ref{lp},
$$\|S\|_{L^2(w)\to L^2(w)}\le C[w]_{A_2}^{3/2},$$
by extrapolation yields yet another proof of the unweighted upper bound  $\|S\|_{L^p\to L^p}\le \frac{C}{(p-1)^{3/2}}, 1<p\le 2$ (see Remark \ref{a2b} below).

Another important ingredient used both in the proofs of Theorems~\ref{lp} and \ref{mm} is Lemma \ref{mes}. This lemma establishes a two-weighted estimate
for the multiplier operator $T_{m\chi_{[a,b]}}$. The need to consider two-weighted estimates comes naturally from the method of the proof of Theorem~\ref{mm}.

The paper is organized as follows. Section 2 contains some preliminaries and, in particular, the proof of Theorem \ref{discr}.
In Section 3 we prove two main technical lemmas. The proof of Theorems \ref{lp} and \ref{mm}
is contained in Section 4. In Section 5 we make several conjectures related to the sharp upper bounds for $\a_p$ and $\b_p$.

\section{Preliminaries}
Although the main objects we deal with are defined on ${\mathbb R}$, the results of subsections 2.1, 2.2 and 2.3 are valid on~${\mathbb R}^n$.
\subsection{Dyadic lattices}
The material of this subsection is taken from~\cite{LN}.

Given a cube $Q_0\subset {\mathbb R}^n$, let ${\mathcal D}(Q_0)$ denote the set of all dyadic cubes with respect to $Q_0$, that is, the cubes
obtained by repeated subdivision of $Q_0$ and each of its descendants into $2^n$ congruent subcubes.

\begin{definition}\label{dl}
A dyadic lattice ${\mathscr D}$ in ${\mathbb R}^n$ is any collection of cubes such that
\begin{enumerate}
\renewcommand{\labelenumi}{(\roman{enumi})}
\item
if $Q\in{\mathscr D}$, then each child of $Q$ is in ${\mathscr D}$ as well;
\item
every 2 cubes $Q',Q''\in {\mathscr D}$ have a common ancestor, i.e., there exists $Q\in{\mathscr D}$ such that $Q',Q''\in {\mathcal D}(Q)$;
\item
for every compact set $K\subset {\mathbb R}^n$, there exists a cube $Q\in {\mathscr D}$ containing $K$.
\end{enumerate}
\end{definition}

In order to construct a dyadic lattice ${\mathscr D}$, it suffices to fix an arbitrary cube $Q_0$ and to expand it dyadically (carefully enough in order to cover the whole space)
by choosing one of $2^n$ possible parents for the top cube and including it into ${\mathscr D}$ together with all its dyadic subcubes during each step.
Therefore, given $h>0$, one can choose a dyadic lattice ${\mathscr D}$ such that for any $Q\in {\mathscr D}$ its sidelength $\ell_Q$ will be of the form $2^kh, k\in {\mathbb Z}$.

\begin{theorem}\label{three}{\rm{(The Three Lattice Theorem)}}
For every dyadic lattice~${\mathscr D}$, there exist $3^n$ dyadic lattices ${\mathscr D}^{(1)},\dots,{\mathscr D}^{(3^n)}$ such that
$$\{3Q: Q\in{\mathscr D}\}=\bigcup_{j=1}^{3^n}{\mathscr D}^{(j)}$$
and for every cube $Q\in {\mathscr D}$ and $j=1,\dots,3^n$, there exists a unique cube $R\in {\mathscr D}^{(j)}$ of
sidelength $\ell_{R}=3\ell_Q$ containing $Q$.
\end{theorem}

\subsection{Some Littlewood-Paley theory}
Denote by ${\mathscr S}({\mathbb R}^n)$ the class of Schwartz functions on ${\mathbb R}^n$.
The following statement can be found in \cite[Lemma 5.12]{FJW} (see also \cite[p. 783]{FJ} for some details).

\begin{lemma}\label{fjw}
There exist $\f,\theta\in {\mathscr S}({\mathbb R}^n)$ satisfying the following properties:
\begin{enumerate}
\renewcommand{\labelenumi}{(\roman{enumi})}
\item
${\rm{supp}}\,\theta\subset\{x:|x|\le 1\}$ and $\int\theta=0$;
\item
${\rm{supp}}\,\widehat \f\subset \{\xi:1/2\le|\xi|\le 2\}$;
\item
$\sum_{k\in {\mathbb Z}}\widehat \f(2^{-k}\xi)\widehat\theta(2^{-k}\xi)\equiv 1$ for all $\xi\not=0$.
\end{enumerate}
\end{lemma}

Property (iii) implies, by taking the Fourier transform, the discrete version of the Calder\'on reproducing formula:
\begin{equation}\label{cald}
f=\sum_{k\in {\mathbb Z}}f*\f_{2^{-k}}*\theta_{2^{-k}}.
\end{equation}

\begin{remark}\label{conv}
There are several interpretations of convergence in (\ref{cald}). In particular, we will use the following one.
Let $1<p<\infty$ and suppose $w\in A_p$. Given $f\in L^p(w)$ and $N\in {\mathbb N}$, set
$$f_N(x)=\sum_{k=-N}^N\int_{E_N}(f*\f_{2^{-k}})(y)\theta_{2^{-k}}(x-y)dy,$$
where $\{E_N\}$ is an increasing sequence of bounded measurable sets such that $E_N\to {\mathbb R}^n$.
Then $f_N\to f$ in $L^p(w)$ as $N\to \infty$. For the continuous version of (\ref{cald}) this fact was proved by
M. Wilson \cite[Th. 7.1]{W2} (see also~\cite{W3}), and in the discrete case the proof follows the same lines.
\end{remark}

The following result is also due to M. Wilson (see \cite[Lemma 2.3]{W1} and \cite[Th. 4.3]{W2}).

\begin{theorem}\label{wil}
Let ${\mathscr D}$ be a dyadic lattice and let ${\mathscr G}\subset {\mathscr D}$ be a finite family of cubes. Assume that
$f=\sum_{Q\in {\mathscr G}}\la_Qa_Q$, where ${\rm{supp}}\,a_Q\subset~Q$, $\|a_Q\|_{L^{\infty}}\le |Q|^{-1/2}$, $\|\nabla a_Q\|_{L^{\infty}}\le \ell_Q^{-1}|Q|^{-1/2}$ and $\int a_Q=0$.
Then for all $1<p<\infty$ and for every $w\in A_p$,
\begin{equation}\label{wils}
\|f\|_{L^p(w)}\le C_{p,n}[w]_{A_p}^{1/2}\Biggl\|\Bigg(\sum_{Q\in {\mathscr G}}\frac{|\la_Q|^2}{|Q|}\chi_Q\Bigg)^{1/2}\Biggr\|_{L^p(w)}.
\end{equation}
\end{theorem}

\begin{remark}\label{Ainf}
Notice that actually (\ref{wils}) was proved in \cite{W1} with a smaller $[w]_{A_{\infty}}$ constant defined by
$$[w]_{A_{\infty}}=\sup_Q\frac{1}{\int_Qw}\int_QM(w\chi_Q),$$
where $Mf(x)=\sup_{Q\ni x}\frac{1}{|Q|}\int_Q|f|$ is the Hardy-Littlewood maximal operator. See also \cite{HP} for various estimates in terms of $[w]_{A_{\infty}}$.
\end{remark}

Theorem \ref{wil} along with the continuous version of (\ref{cald}) was applied in~\cite{W1} in order to obtain the $L^p(w)$-norm relation between $f$ and the continuous square function.
In a similar way, using (\ref{cald}), we obtain the $L^p(w)$-norm relation between $f$ and the discrete square function defined
(for a given dyadic lattice ${\mathscr D}$) by
$$S_{\f,{\mathscr D}}(f)(x)=\left(\sum_{k\in {\mathbb Z}}\sum_{Q\in {\mathscr D}:\ell_Q=2^{-k}}\Big(\frac{1}{|Q|}\int_Q|f*\f_{2^{-k}}|^2\Big)\chi_Q(x)\right)^{1/2}.$$

\begin{theorem}\label{discr}
There exists a function $\f\in {\mathscr S}({\mathbb R}^n)$ with ${\rm{supp}}\,\widehat \f\subset \{\xi:1/2\le|\xi|\le 2\}$ and there are $3^n$ dyadic lattices ${\mathscr D}^{(j)}$
such that for every $w\in A_p$ and for any $f\in L^p(w), 1<p<\infty,$
$$\|f\|_{L^p(w)}\le C_{p,n}[w]_{A_p}^{1/2}\sum_{j=1}^{3^n}\|S_{\f,{\mathscr D}^{(j)}}(f)\|_{L^p(w)}.$$
\end{theorem}

\begin{proof}
Let $\f,\theta$ be functions from Lemma \ref{fjw}. Let ${\mathscr D}$ be a dyadic lattice such that
for every $Q\in {\mathscr D}$ its sidelength is of the form $\ell_Q=\frac{2^k}{3}, k\in~{\mathbb Z}$.
Let ${\mathscr D}^{(j)}, j=1,\dots, 3^n,$ be dyadic lattices obtained by applying Theorem~\ref{three} to ${\mathscr D}$. Then for every $Q\in {\mathscr D}^{(j)}$
its sidelength is of the form $\ell_Q=2^k, k\in {\mathbb Z}$.

For $Q\in {\mathscr D}$ with $\ell_Q=2^{-k}/3$ set
$$\gamma_Q(x)=\int_{Q}(f*\f_{2^{-k}})(y)\theta_{2^{-k}}(x-y)dy.$$
It is easy to check that ${\rm{supp}}\,\ga_Q\subset 3Q$, $\int \ga_Q=0$ and
\begin{equation}\label{cond}
\max(\|\ga_Q\|_{L^{\infty}}, \ell_Q\|\nabla \ga_Q\|_{L^{\infty}})\le c\left(\frac{1}{|Q|}\int_Q|f*\f_{2^{-k}}|^2\right)^{1/2},
\end{equation}
where $c$ depends only on $n$ and $\theta$.

Take an increasing sequence of cubes $Q_N\in {\mathscr D}$ such that $\ell_{Q_N}=\frac{2^N}{3}, N\in {\mathbb N}$.
Set
$${\mathscr G}_N=\{Q\in {\mathscr D}:Q\subseteq Q_N, \ell_Q=2^{-k}/3, k\in [-N,N]\}.$$
By Theorem \ref{three}, one can write
$$\{3Q:Q\in {\mathscr G}_N\}=\bigcup_{j=1}^{3^n}{\mathscr G}^{(j)}_N,$$
where ${\mathscr G}^{(j)}_N\subset {\mathscr D}^{(j)}$. Then
\begin{eqnarray*}
f_N(x)&=&\sum_{k=-N}^N\int_{Q_N}(f*\f_{2^{-k}})(y)\theta_{2^{-k}}(x-y)dy\\
&=&\sum_{k=-N}^N\sum_{Q\in {\mathscr D}: Q\subseteq Q_N, \ell_Q=2^{-k}/3}\gamma_Q(x)=\sum_{j=1}^{3^n}\sum_{P\in {\mathscr G}^{(j)}_N}\la_P^{(j)}a_P^{(j)},
\end{eqnarray*}
where, for $P=3Q, Q\in {\mathscr D}, \ell_Q=2^{-k}/3,$ we set
$$\la_P^{(j)}=c\Big(\int_{3Q}|f*\f_{2^{-k}}|^2\Big)^{1/2}$$
and $a_P^{(j)}=\frac{1}{\la_P^{(j)}}\ga_Q$.

By (\ref{cond}), we have that the functions $a_P^{(j)}$ satisfy all conditions from Theorem \ref{wil}. Therefore, by (\ref{wils}),
\begin{eqnarray*}
\|f_N\|_{L^p(w)}&\le& C_{p,n}[w]_{A_p}^{1/2}\sum_{j=1}^{3^n}
\Biggl\|\Bigg(\sum_{P\in {\mathscr G}^{(j)}_N}\frac{|\la_P^{(j)}|^2}{|P|}\chi_P\Bigg)^{1/2}\Biggr\|_{L^p(w)}\\
&\le& C_{p,n}[w]_{A_p}^{1/2}\sum_{j=1}^{3^n}\|S_{\f,{\mathscr D}^{(j)}}(f)\|_{L^p(w)}.
\end{eqnarray*}
Applying the convergence argument as described in Remark \ref{conv} completes the proof.
\end{proof}

\subsection{The sharp extrapolation}
The following result was proved in~\cite{D2}.

\begin{theorem}\label{shext}
Assume that for some $f,g$ and for all weights $w\in A_{p_0}$,
$$\|f\|_{L^{p_0}(w)}\le CN([w]_{A_{p_0}})\|g\|_{L^{p_0}(w)},$$
where $N$ is an increasing function and the constant $C$ does not depend on $w$. Then for all $1<p<\infty$ and all $w\in A_p$,
$$\|f\|_{L^p(w)}\le CK(w)\|g\|_{L^p(w)},$$
where
$$
K(w)=\begin{cases}
N\big([w]_{A_p}(2\|M\|_{L^p(w)\to L^p(w)})^{p_0-p}\big),& {\rm{if}}\,\, p<p_0;\\
N\Big([w]_{A_p}^{\frac{p_0-1}{p-1}}(2\|M\|_{L^{p'}(w^{1-p'})\to L^{p'}(w^{1-p'})})^{\frac{p-p_0}{p-1}}\Big), & {\rm{if}}\,\, p>p_0.
\end{cases}
$$
In particular, $K(w)\le C_1N\Big(C_2[w]_{A_p}^{\max\big(1,\frac{p_0-1}{p-1}\big)}\Big)$ for $w\in A_p$.
\end{theorem}

\subsection{Some two-weighted estimates}
Let
$$Hf(x)=\text{p.v.}\frac{1}{\pi}\int_{\mathbb R}\frac{f(y)}{x-y}dy\quad\text{and}\quad H^{\star}f(x)=\sup_{\e>0}\frac{1}{\pi}\left|\int_{|x-y|>\e}\frac{f(y)}{x-y}dy\right|$$
be the Hilbert and the maximal Hilbert transforms, respectively.

Given two weights $u$ and $v$, set
$$[u,v]_{A_2}=\sup_{Q}\,\l u\r_Q\l v^{-1}\r_Q.$$
Then the following two-weighted estimates hold:
\begin{eqnarray}
&&\max(\|M\|_{L^2(v)\to L^2(u)},\|H^{\star}\|_{L^2(v)\to L^2(u)})\label{two}\\
&&\le C[u,v]_{A_2}^{1/2}\big([u]_{A_2}^{1/2}+[v]_{A_2}^{1/2}\big).\nonumber
\end{eqnarray}
The proofs of these estimates can be found in \cite{HL,HP} (notice that stronger versions of (\ref{two}) in terms of the $[w]_{A_{\infty}}$ constants are proved there).

\subsection{The partial sum operator}
Given an interval $[a,b]$, the partial sum operator $S_{[a,b]}$ is defined by $\widehat{S_{[a,b]}f}=\widehat f\chi_{[a,b]}$. We will use two standard facts about $S_{[a,b]}$
(see, e.g., \cite{D1}).
First,
\begin{equation}\label{ff}
S_{[a,b]}=\frac{i}{2}({\mathcal M}_aH{\mathcal M}_{-a}-{\mathcal M}_bH{\mathcal M}_{-b}),
\end{equation}
where ${\mathcal M}_af(x)=e^{2\pi iax}f(x)$.
Second, if $(T_{m\chi_{[a,b]}}f)\,\widehat{}=m\chi_{[a,b]}\widehat f$, then
\begin{equation}\label{sf}
T_{m\chi_{[a,b]}}f=m(a)S_{[a,b]}f+\int_a^b(S_{[t,b]}f)m'(t)dt.
\end{equation}

\section{Two key lemmas}
Given a dyadic lattice ${\mathscr D}$ in ${\mathbb R}$, a weight $w$ and $k\in {\mathbb Z}$, denote
$$w_{k,{\mathscr D}}=\sum_{I\in {\mathscr D}:|I|=2^{-k}}\l w\r_{I}\chi_I.$$

\begin{lemma}\label{a2} Let $w\in A_2$. Then $w_{k,{\mathscr D}}\in A_2$ and
\begin{equation}\label{a21}
[w_{k,{\mathscr D}}]_{A_2}\le 9[w]_{A_2}.
\end{equation}
Also, for two arbitrary dyadic lattices ${\mathscr D}$ and ${\mathscr D}'$,
\begin{equation}\label{a22}
[w_{k,{\mathscr D}},((w^{-1})_{k,{\mathscr D}'})^{-1}]_{A_2}\le 9[w]_{A_2}.
\end{equation}
\end{lemma}

\begin{proof} Denote $u=w_{k,{\mathscr D}}$ and ${\mathcal P}_k=\{I\in {\mathscr D}:|I|=2^{-k}\}$.
Take an arbitrary interval $J\subset {\mathbb R}$.
Notice that
\begin{equation}\label{uj}
\l u\r_J=\frac{1}{|J|}\sum_{I\in {\mathcal P}_k:I\cap J\not=\emptyset}\frac{|I\cap J|}{|I|}\int_Iw.
\end{equation}
Next, by H\"older's inequality,
$$|I|^2\le \Big(\int_Iw\Big)\Big(\int_Iw^{-1}\Big),$$
which implies
\begin{eqnarray}
\l u^{-1}\r_J&=&\frac{1}{|J|}\sum_{I\in {\mathcal P}_k:I\cap J\not=\emptyset}|I\cap J|\frac{|I|}{\int_Iw}\nonumber\\
&\le& \frac{1}{|J|}\sum_{I\in {\mathcal P}_k:I\cap J\not=\emptyset}\frac{|I\cap J|}{|I|}\int_Iw^{-1}.\label{uj1}
\end{eqnarray}

Denote
$$J^*=\bigcup_{I\in {\mathcal P}_k:I\cap J\not=\emptyset}I.$$
If $|J|>2^{-k}$, then $|J^*|\le 3|J|$, and hence, by (\ref{uj}) and (\ref{uj1}),
\begin{equation}\label{big}
\l u\r_J\le \frac{1}{|J|}\int_{J^*}w\le 3\,\l w\r_{J^*}\quad\text{and}\quad \l u^{-1}\r_J\le  3\,\l w^{-1}\r_{J^*}.
\end{equation}
Assume that $|J|\le 2^{-k}$. Then $|J^*|\le 2^{-k+1}$. Hence in this case,
$$\l u\r_J\le \frac{1}{|I|}\int_{J^*}w\le 2\,\l w\r_{J^*}\quad\text{and}\quad \l u^{-1}\r_J\le  2\,\l w^{-1}\r_{J^*},$$
which along with (\ref{big}) implies (\ref{a21}).

The proof of (\ref{a22}) is identically the same. Denote $v=((w^{-1})_{k,{\mathscr D}'})^{-1}$.
If $|J|>2^{-k}$, then by (\ref{big}),
$$
\l u\r_J\le 3\,\l w\r_{J^*}\quad\text{and}\quad \l v^{-1}\r_J\le  3\,\l w^{-1}\r_{J^*}.
$$
Similarly, if $|J|\le 2^{-k}$, then
$$
\l u\r_J\le 2\,\l w\r_{J^*}\quad\text{and}\quad \l v^{-1}\r_J\le  2\,\l w^{-1}\r_{J^*},
$$
which along with the previous estimate proves (\ref{a22}).
\end{proof}

Define the operator $T_{m\chi_{[a,b]}}$ by $(T_{m\chi_{[a,b]}}f)\,\widehat{}=m\chi_{[a,b]}\widehat f$. In the lemma below we use
the same notation $u_{k,{\mathscr D}}$ as in Lemma \ref{a2}.

\begin{lemma}\label{mes}
Assume that $m$ is a bounded and differentiable function on $[a,b]$. Then for all $u,v\in A_2$,
$$\|T_{m\chi_{[a,b]}}f\|_{L^2(u_{k,{\mathscr D}})}\le cK(m)N(u,v)(2^{-k}(b-a)+1)\|f\|_{L^2(v)},$$
where $K(m)=\|m\|_{L^{\infty}}+\int_a^b|m'(t)|dt$,
$$N(u,v)=\min\big([u,v]_{A_2},[u_{k,{\mathscr D}},v]_{A_2}\big)^{1/2}\big([u]_{A_2}^{1/2}+[v]_{A_2}^{1/2}\big)$$
and $c>0$ is an absolute constant.
\end{lemma}

\begin{proof}
Let $t\in [a,b)$. Take an arbitrary $I\in {\mathscr D}$ with $|I|=2^{-k}$. Notice that
$$\|S_{[t,b]}f\|_{L^{\infty}}\le (b-a)\|f\|_{L^1}.$$
Therefore, for all $x,y\in I$,
\begin{eqnarray}
&&|S_{[t,b]}f(y)|\le (b-a)\int_{3I}|f|+|S_{[t,b]}(f\chi_{{\mathbb R}\setminus 3I})(y)|\label{s2t}\\
&&\le 3(b-a)2^{-k}Mf(x)+|S_{[t,b]}(f\chi_{{\mathbb R}\setminus 3I})(y)|.\nonumber
\end{eqnarray}

Applying (\ref{ff}) yields
\begin{eqnarray}
|S_{[t,b]}(f\chi_{{\mathbb R}\setminus 3I})(y)|&\le& |H{\mathcal M}_{-t}(f\chi_{{\mathbb R}\setminus 3I})(y)|\label{repest}\\
&+&|H{\mathcal M}_{-b}(f\chi_{{\mathbb R}\setminus 3I})(y)|.\nonumber
\end{eqnarray}
For every $t\in [a,b]$,
\begin{eqnarray}
&&|H{\mathcal M}_{-t}(f\chi_{{\mathbb R}\setminus 3I})(y)-H{\mathcal M}_{-t}(f\chi_{{\mathbb R}\setminus 3I})(x)|\label{hm}\\
&&\le c|I|\int_{{\mathbb R}\setminus 3I}|f(\xi)|\frac{1}{|x-\xi|^2}d\xi\le cMf(x).\nonumber
\end{eqnarray}
Further,
\begin{eqnarray*}
|H{\mathcal M}_{-t}(f\chi_{{\mathbb R}\setminus 3I})(x)|&\le& |H{\mathcal M}_{-t}(f\chi_{{\mathbb R}\setminus [x-|I|/2,x+|I|/2]})(x)|\\
&+&|H{\mathcal M}_{-t}(f\chi_{3I\setminus [x-|I|/2,x+|I|/2]})(x)|\\
&\le& H^{\star}{\mathcal M}_{-t}f(x)+cMf(x),
\end{eqnarray*}
which, combined with (\ref{s2t}), (\ref{repest}) and (\ref{hm}), implies
$$|S_{[t,b]}f(y)|\le H^{\star}{\mathcal M}_{-b}f(x)+H^{\star}{\mathcal M}_{-t}f(x)+(3(b-a)2^{-k}+c)Mf(x).$$

From this and from (\ref{sf}), for all $x,y\in I$ we have
$$|T_{m\chi_{[a,b]}}f(y)|\le cK(m){\mathcal T}(f)(x)+\int_{a}^{b}H^{\star}{\mathcal M}_{-t}f(x)|m'(t)|dt,$$
where
$${\mathcal T}(f)(x)=H^{\star}{\mathcal M}_{-b}f(x)+H^{\star}{\mathcal M}_{-a}f(x)+(2^{-k}(b-a)+1)Mf(x).$$
Therefore,
\begin{equation}\label{kest}
\frac{1}{|I|}\int_{I}|T_{m\chi_{[a,b]}}f|^2\le \inf_I\Big(cK(m){\mathcal T}(f)+\int_{a}^{b}H^{\star}{\mathcal M}_{-t}f|m'(t)|dt\Big)^2.
\end{equation}

Hence, applying Minkowski's inequality and using (\ref{two}), we obtain
\begin{eqnarray*}
&&\|T_{m\chi_{[a,b]}}f\|_{L^2(u_{k,{\mathscr D}})}\le \Big\|cK(m){\mathcal T}(f)+\int_{a}^{b}H^{\star}{\mathcal M}_{-t}f|m'(t)|dt\Big\|_{L^2(u)}\\
&&\le cK(m)\|{\mathcal T}(f)\|_{L^2(u)}+\int_{a}^{b}\|H^{\star}{\mathcal M}_{-t}f\|_{L^2(u)}|m'(t)|dt\\
&&\le cK(m)(2^{-k}(b-a)+1)[u,v]_{A_2}^{1/2}([u]_{A_2}^{1/2}+[v]_{A_2}^{1/2})\|f\|_{L^2(v)}.
\end{eqnarray*}

On the other hand, (\ref{kest}) also implies
$$\|T_{m\chi_{[a,b]}}f\|_{L^2(u_{k,{\mathscr D}})}\le \Big\|cK(m){\mathcal T}(f)+\int_{a}^{b}H^{\star}{\mathcal M}_{-t}f|m'(t)|dt\Big\|_{L^2(u_{k,{\mathscr D}})}.$$
Therefore, by the previous arguments and Lemma \ref{a2},
\begin{eqnarray*}
&&\|T_{m\chi_{[a,b]}}f\|_{L^2(u_{k,{\mathscr D}})}\\
&&\le cK(m)(2^{-k}(b-a)+1)[u_{k,{\mathscr D}},v]_{A_2}^{1/2}([u]_{A_2}^{1/2}+[v]_{A_2}^{1/2})\|f\|_{L^2(v)},
\end{eqnarray*}
which completes the proof.
\end{proof}

\section{Proof of Theorems \ref{lp} and \ref{mm}}
The lower bounds for $\a_p$ and $\b_p$ are explained in the Introduction. Therefore, we are left with establishing the upper bounds.

\begin{proof}[Proof of Theorem \ref{lp}]
By duality, the estimate
\begin{equation}\label{sfl}
\|Sf\|_{L^p(w)}\le C[w]_{A_p}^{\frac{1}{2}\frac{1}{p-1}+\max\big(1,\frac{1}{p-1}\big)}\|f\|_{L^p(w)}
\end{equation}
is equivalent to
$$
\left\|\sum_{k\in {\mathbb Z}}S_{\D_k}\psi_k\right\|_{L^{p'}(\si)}\le C[\si]_{A_{p'}}^{\frac{1}{2}+\max(1,p-1)}\left\|\Big(\sum_{k\in {\mathbb Z}}|\psi_k|^2\Big)^{1/2}\right\|_{L^{p'}(\si)},
$$
where $\si=w^{1-p'}$. Changing here $p'$ by $p$ and $\si$ by $w$, we see that it suffices to prove that
\begin{equation}\label{suf}
\left\|\sum_{k\in {\mathbb Z}}S_{\D_k}\psi_k\right\|_{L^{p}(w)}\le C[w]_{A_{p}}^{\frac{1}{2}+\max\big(1,\frac{1}{p-1}\big)}\left\|\Big(\sum_{k\in {\mathbb Z}}|\psi_k|^2\Big)^{1/2}\right\|_{L^{p}(w)}.
\end{equation}

Applying Theorem \ref{discr} yields
$$
\left\|\sum_{k\in {\mathbb Z}}S_{\D_k}\psi_k\right\|_{L^{p}(w)}\le C[w]_{A_{p}}^{\frac{1}{2}}
\sum_{j=1}^{3}\left\|S_{\f,{\mathscr D}^{(j)}}\Big(\sum_{k\in {\mathbb Z}}S_{\D_k}\psi_k\Big)\right\|_{L^p(w)}.
$$
Therefore, by Theorem \ref{shext}, (\ref{suf}) will follow from
\begin{equation}\label{wf}
\left\|S_{\f,{\mathscr D}}\Big(\sum_{k\in {\mathbb Z}}S_{\D_k}\psi_k\Big)\right\|_{L^2(w)}\le C[w]_{A_2}\left\|\Big(\sum_{k\in {\mathbb Z}}|\psi_k|^2\Big)^{1/2}\right\|_{L^{2}(w)}.
\end{equation}

Using that ${\rm{supp}}\,\widehat{\f_{2^{-k}}}\subset \{\xi:2^{k-1}\le|\xi|\le 2^{k+1}\}$, we have
$$
\Big(\sum_{j\in {\mathbb Z}}S_{\D_j}\psi_j\Big)*\f_{2^{-k}}=(S_{\D_{k-1}}\psi_{k-1}+S_{\D_{k}}\psi_{k})*\f_{2^{-k}},
$$
which implies
\begin{eqnarray*}
&&S_{\f,{\mathscr D}}\Big(\sum_{j\in {\mathbb Z}}S_{\D_j}\psi_j\Big)(x)^2\\
&&=
\sum_{k\in{\mathbb Z}}\sum_{I\in {\mathscr D}:\ell_I=2^{-k}}\left(\frac{1}{|I|}\int_I|(S_{\D_{k-1}}\psi_{k-1}+S_{\D_{k}}\psi_{k})*\f_{2^{-k}}|^2\right)\chi_I(x).
\end{eqnarray*}
Hence, in order to prove (\ref{wf}), it suffices to establish that for every $k\in {\mathbb Z}$,
\begin{equation}\label{es1}
\|(S_{\D_{k-1}}f)*\f_{2^{-k}}\|_{L^2(w_{k,{\mathscr D}})}\le C[w]_{A_2}\|f\|_{L^2(w)}
\end{equation}
and
\begin{equation}\label{es2}
\|(S_{\D_{k}}f)*\f_{2^{-k}}\|_{L^2(w_{k,{\mathscr D}})}\le C[w]_{A_2}\|f\|_{L^2(w)}.
\end{equation}

Since
$$((S_{\D_{k-1}}f)*\f_{2^{-k}})\,\widehat{}\,(\xi)=\widehat\f(2^{-k}\xi)\chi_{\{2^{k-1}\le |\xi|\le 2^k\}}\widehat f(\xi),$$
(\ref{es1}) is an immediate corollary of Lemma \ref{mes} (applied in the case of equal weights). Estimate (\ref{es2}) follows in the same way.
Notice that the constants~$C$ in (\ref{es1}) and (\ref{es2}) can be taken as
$$C=c\left(\|\widehat \f\|_{L^{\infty}}+\int_{1/2\le |\xi|\le 2}|(\widehat \f)'(\xi)|d\xi\right)$$
with some absolute $c>0$.
\end{proof}

\begin{remark}\label{expl}
There is a minor inaccuracy in the proof, namely, applying Theorem \ref{discr}, we have used that $\sum_{k\in {\mathbb Z}}S_{\D_k}\psi_k\in L^p(w)$ as an {\it a priori} assumption. This point can be fixed in several ways.
First, by \cite{K}, $f\in L^p(w)$ implies $Sf\in L^p(w)$ for $w\in A_p$ for all $1<p<\infty$. By duality, this means that $\Big(\sum_{k\in {\mathbb Z}}|\psi_k|^2\Big)^{1/2}\in L^p(w)$ implies $\sum_{k\in {\mathbb Z}}S_{\D_k}\psi_k\in L^p(w)$.

However, one can avoid the use of \cite{K} as follows. Defining
$$S_Nf=\left(\sum_{k=-N}^N|S_{\D_k}f|^2\right)^{1/2},$$
we have that (\ref{sfl}) with $S_Nf$ instead of $Sf$ is equivalent to (\ref{suf}) with $\sum_{k=-N}^{N}S_{\D_k}\psi_k$ on the left-hand side. But the fact that $\sum_{k=-N}^{N}S_{\D_k}\psi_k\in L^p(w)$
follows immediately from (\ref{ff}). The rest of the proof is exactly the same, and we obtain (\ref{sfl}) with $S_Nf$ instead of $Sf$ with the corresponding constant independent of $N$. Letting $N\to \infty$ yields
the desired bound for $S$.
\end{remark}

\begin{remark}\label{a2b}
Theorem \ref{lp} in the case $p=2$ says that
$$\|S\|_{L^2(w)\to L^2(w)}\le C[w]_{A_2}^{3/2}.$$
From this, by Theorem \ref{shext},
$$\|S\|_{L^p\to L^p}\le C\|M\|_{L^p\to L^p}^{3/2}\quad(1<p\le 2).$$
Since $\|M\|_{L^p\to L^p}\simeq \frac{1}{p-1}$ for $1<p\le 2$, we obtain the sharp upper bound
$$\|S\|_{L^p\to L^p}\le \frac{C}{(p-1)^{3/2}}\quad(1<p\le 2)$$
found by J. Bourgain \cite{B}.
\end{remark}

\begin{proof}[Proof of Theorem \ref{mm}]
Using the fact that
$$\|T_m\|_{L^p(w)\to L^p(w)}=\|T_m\|_{L^{p'}(\si)\to L^{p'}(\si)}$$
and $[\si]_{A_{p'}}=[w]_{A_p}^{\frac{1}{p-1}}$,
it suffices to prove that
\begin{equation}\label{stp}
\|T_m\|_{L^p(w)\to L^p(w)}\le C_{p,m}[w]_{A_p}^{\frac{1}{2}+\frac{3}{2}\frac{1}{p-1}}\quad(1<p\le 2).
\end{equation}
By Theorems \ref{discr} and \ref{shext}, (\ref{stp}) will follow from
\begin{equation}\label{wff}
\|S_{\f,{\mathscr D}}(T_mf)\|_{L^2(w)}\le C_m[w]_{A_2}^{3/2}\|f\|_{L^2(w)}.
\end{equation}

Notice that
$$
\|S_{\f,{\mathscr D}}(T_mf)\|_{L^2(w)}=\left(\sum_{k\in {\mathbb Z}}\int_{{\mathbb R}}|(T_mf)*\f_{2^{-k}}|^2w_{k,{\mathscr D}}dx\right)^{1/2}.
$$
Therefore, by duality, (\ref{wff}) is equivalent to
$$\left\|\sum_{k\in {\mathbb Z}}(T_m\psi_k)*\f_{2^{-k}}\right\|_{L^2(w^{-1})}\le C_m[w]_{A_2}^{3/2}\left(\sum_{k\in {\mathbb Z}}\int_{{\mathbb R}}|\psi_k|^2(w_{k,{\mathscr D}})^{-1}dx\right)^{1/2}.$$

Applying Theorem \ref{discr} again, we see that the question is reduced to the estimate
\begin{eqnarray}
&&\left(\sum_{k\in {\mathbb Z}}\|(\sum_{j\in {\mathbb Z}}(T_m\psi_j)*\f_{2^{-j}})*\f_{2^{-k}}\|^2_{L^2((w^{-1})_{k,{\mathscr D}'})}\right)^{1/2}\label{rte}\\
&&\le C_m[w]_{A_2}\left(\sum_{k\in {\mathbb Z}}\|\psi_k\|_{L^2((w_{k,{\mathscr D}})^{-1})}^2\right)^{1/2}\nonumber
\end{eqnarray}
for some dyadic lattices ${\mathscr D}$ and ${\mathscr D}'$.

Since
$$
(\sum_{j\in {\mathbb Z}}(T_m\psi_j)*\f_{2^{-j}})*\f_{2^{-k}}=\sum_{j=k-1}^{k+1}(T_m\psi_j)*\f_{2^{-j}}*\f_{2^{-k}},
$$
in order to prove (\ref{rte}), it suffices to show that for every $k\in {\mathbb Z}$ and every $j=k-1,k,k+1$,
\begin{equation}\label{stst}
\|(T_mf)*\f_{2^{-j}}*\f_{2^{-k}}\|_{L^2((w^{-1})_{k,{\mathscr D}'})}\le C_m[w]_{A_2}\|f\|_{L^2((w_{k,{\mathscr D}})^{-1})}.
\end{equation}

By Lemma \ref{a2},
$$
[(w^{-1})_{k,{\mathscr D}'},(w_{k,{\mathscr D}})^{-1})]_{A_2}^{1/2}\big([(w^{-1})_{k,{\mathscr D}'}]_{A_2}^{1/2}+[(w_{k,{\mathscr D}})^{-1})]_{A_2}^{1/2}\big)\le c[w]_{A_2}.
$$
From this and from Lemma \ref{mes} we obtain (\ref{stst}) with
$$C_m=cC_{\f}\Big(\|m\|_{L^{\infty}}+\sup_{k\in {\mathbb Z}}\int_{\D_k}|m'(t)|dt\Big),$$
which completes the proof.
\end{proof}

\begin{remark}\label{jmm}
As in Remark \ref{expl}, it is not difficult to justify the use of Theorem \ref{discr}. We omit the details.
\end{remark}

\section{Concluding remarks}
\subsection{On the sharpness of $\a_p$ and $\b_p$} The extrapolation principle explained in the Introduction says
that if $\xi_p$ is the best possible exponent in $\|T\|_{L^p(w)\to L^p(w)}\le C[w]_{A_p}^{\xi_p}$, then
$\xi_p\ge\max(\ga_2,\frac{\ga_1}{p-1})$, where $\ga_1$ and~$\ga_2$ are the constants appearing in the
endpoint asymptotic relations for $\|T\|_{L^p\to L^p}$. In fact, for many particular operators we have that
$\xi_p=\max(\ga_2,\frac{\ga_1}{p-1})$.

Therefore, it is plausible that the upper bounds for $\a_p$ and $\b_p$ from Theorems \ref{lp} and \ref{mm}
are not sharp for $p>2$ and $1<p<\infty$, respectively, and it is natural to make the following.

\begin{con}\label{clp}
The best possible exponent $\a_p$ in
$$\|S\|_{L^p(w)\to L^p(w)}\le C_p[w]_{A_p}^{\a_p}$$
is
$$\a_p=\max\left(1,\frac{3}{2}\frac{1}{p-1}\right)\quad(1<p<\infty).$$
\end{con}

\begin{con}\label{cmm}
The best possible exponent $\b_p$ in
$$\|T_m\|_{L^p(w)\to L^p(w)}\le C_{p,m}[w]_{A_p}^{\a_p}$$
is
$$\b_p=\frac{3}{2}\max\left(1,\frac{1}{p-1}\right)\quad(1<p<\infty).$$
\end{con}

Observe that by Theorem \ref{shext}, in order to establish Conjectures \ref{clp} and \ref{cmm}, it suffices to show that
$$\|S\|_{L^{5/2}(w)\to L^{5/2}(w)}\le C[w]_{A_{5/2}}\quad\text{and}\quad
\|T_m\|_{L^2(w)\to L^2(w)}\le C_{m}[w]_{A_2}^{3/2},
$$
respectively.

\subsection{Sparse bounds for $S$ and $T_m$?} A family of cubes ${\mathcal S}$ is called sparse if there exist $0<\eta<1$ and a family of pairwise disjoint sets $\{E_Q\}_{Q\in {\mathcal S}}$
such that $E_Q\subset Q$ and $|E_Q|\ge\eta |Q|$ for all $Q\in {\mathcal S}$. By a sparse bound for a given operator $T$ we mean an estimate of the form
$$|\langle Tf,g \rangle|\le C\sum_{Q\in {\mathcal S}}\langle f\rangle_{r,Q}\langle g\rangle_{s,Q}|Q|,$$
with suitable $1\le r,s<\infty$, where $\langle f\rangle_{p,Q}=\l|f|^p\r_Q^{1/p},$
and ${\mathcal S}$ is a sparse family.

Sparse bounds have become a powerful tool for obtaining sharp quantitative weighted estimates in recent years (see, e.g., \cite{BFP,CCPO,L2}). Therefore it would be natural to try to attack Conjectures
\ref{clp} and \ref{cmm} by means of the corresponding sparse bounds for $S$ and $T_m$.

At this point, we mention that it is not clear to us what is the sparse bound for $S$ leading to Conjecture \ref{clp}.
For example, it is plausible that~$S$ satisfies 
$$|\langle Sf,g \rangle|\le \frac{C}{(r-1)^{1/2}}\sum_{Q\in {\mathcal S}}\langle f\rangle_{r,Q}\langle g\rangle_{1,Q}|Q|\quad(1<r\le 2)$$
but one can show that this estimate leads to the same upper bound for $\a_p$ as obtained in Theorem \ref{lp}.

Contrary to this, the sparse bound
\begin{equation}\label{ssb}
|\langle T_mf,g \rangle|\le \frac{C}{(r-1)^{1/2}}\sum_{Q\in {\mathcal S}}\langle f\rangle_{r,Q}\langle g\rangle_{r,Q}|Q|\quad(1<r\le 2)
\end{equation}
would imply Conjecture \ref{cmm}. The technique developed in \cite{TW} probably may play an important role in establishing (\ref{ssb}).

\end{document}